\documentclass[12pt, twoside]{article}

\usepackage{amssymb,amsmath,amsthm}
\usepackage{enumerate, enumitem}
\usepackage{color}
\usepackage[colorlinks=true]{hyperref}

\def\titlerunning#1{\gdef\titrun{#1}}

\makeatletter

\def\author#1{
    \gdef\autrun{
        \def\and{\unskip, }#1
    }
    \gdef\@author{#1}
}

\def\email#1{Email: #1}
\def\emails#1{Emails: #1}

\makeatother

\def\keywords#1{
    \par\medskip
    \noindent\textbf{Keywords.} #1
}
\def\subjclass#1{{
    \renewcommand{\thefootnote}{}%
    \footnote{
        \emph{Mathematics Subject Classification (2020):} #1
    }
}}
\newtheorem{theorem}{Theorem}[section]

\newtheorem{lemma}[theorem]{Lemma}
\newtheorem{proposition}[theorem]{Proposition}
\newtheorem{corollary}[theorem]{Corollary}

\theoremstyle{definition}
\newtheorem{definition}[theorem]{Definition}
\newtheorem{remark}[theorem]{Remark}
\newtheorem{example}[theorem]{Example}

\hypersetup{linkcolor=blue,urlcolor=red,citecolor=red}
\numberwithin{equation}{section}

\begin{document}
\titlerunning{}
\title{Sampling Observability for Heat Equations with Memory }

\author{
LingyingMa\footnotemark[2]
 \and
Gengsheng Wang\footnotemark[2]
 \and
Yubiao Zhang\footnotemark[2] \footnotemark[3] \footnotemark[1]
}

\renewcommand{\thefootnote}{\fnsymbol{footnote}}

\footnotetext[1]{
    Corresponding author. 
    \email{\texttt{yubiao\b{ }zhang@yeah.net}}
}
\footnotetext[2]{
    Center for Applied Mathematics and KL-AAGDM, 
    Tianjin University,
    300072 Tianjin,  China. 
    \emails{\texttt{ly\_ma@tju.edu.cn}, \texttt{wanggs62@yeah.net}} 
}
\footnotetext[3]{
    Chair for Dynamics, Control, Machine Learning  and Numerics, 
    Department of Mathematics,  
    Friedrich-Alexander-Universit\"{a}t Erlangen-N\"{u}rnberg,
    91058 Erlangen, Germany. 
}


\date{}
\maketitle
\begin{abstract}
     This paper studies the sampling observability for the heat equations with  memory in the lower-order term, where the observation is conducted at a finite number of time instants and on a small open subset at each time instant.
        We  present a two-sided sampling observability inequality and
     give a sharp sufficient  condition   to ensure
    the aforementioned inequality. We  also provide a method   to select the time instants and then to design the observation regions, based on a given memory kernel, such that the above-mentioned inequality holds for these time instants and  observation regions. 
    Additionally, we demonstrate that the positions of these time instants depend significantly on the memory kernel.
\end{abstract}




\section{Introduction}

 Let  $\Omega\subset\mathbb R^n$ ($n\in \mathbb N^+:=\{1,2,3,\dots\}$)  be a bounded domain with its $C^2$-boundary $\partial\Omega$.
 Consider  the   heat equation with  memory:
\begin{eqnarray}\label{20240215-OriginalEquation}
\begin{cases}
    \partial_ty(t,x) -  \Delta y(t,x) + \int_0^t M(t-s) y(s,x) ds=0,                          ~(t,x)\in\mathbb{R}^+\times\Omega,\\
    y(t,x)=0, ~(t,x)\in\mathbb{R}^+\times\partial\Omega,\\
    y(0,x) = y_0(x),~x\in\Omega,
\end{cases}
\end{eqnarray}
where $y_0$ is the initial datum taken from $L^2(\Omega)$, $\mathbb{R}^+:=(0,+\infty)$ and $M$ is a memory kernel, which satisfies the following assumption:
\begin{itemize}
	\item[]  $(\mathcal{A})$ The time-dependent memory kernel $M$ belongs to the space $C^2([0,+\infty))$.
\end{itemize}
By a similar method used in \cite[Theorem 1.2 on Page 184]{Pazy-Springer-1983}, the equation \eqref{20240215-OriginalEquation} has a unique solution  $y(\cdot;y_0)\in C([0,+\infty);L^2(\Omega))$.

The aim of this paper is
to study the observability for the equation \eqref{20240215-OriginalEquation}, where the observation is conducted at a finite number of time instants
$t_j$ ($j=1,\dots,m$, with $m\in \mathbb{N}^+$), and on $\omega_j$ at each time $t_j$, where $\omega_j$ is an open subset of $\overline{\Omega}$.
We will present a two-sided sampling observability inequality and
 give a sharp sufficient  condition   on $\{t_j\}_{j=1}^m$, $\{\omega_j\}_{j=1}^m$ and $M$ to ensure
the aforementioned inequality. We will also provide a  way to select the time instants and then  design the observation regions, based on a given memory kernel.

The observability is an important subject in  control theory, as well as the inverse problems. There have been much
 literature on this issue. We refer to \cite[Section 1.4, pp. 11-14]{TucsnakWeiss-2009} for the observability of ODEs,  \cite{AEWZ-JEMS-2014, FZ-AIHPANL-2000,FI-LectureNotes-1996} and \cite{BLR-SICON-1992, Zhang-SICON-2000} for PDEs.

 Compared with the understanding of the observability for PDEs without memory, that for  PDEs with memory is still very immature.
 Due to the involvement of the history values, the latter is more difficult.
 For proper time-invariant observation regions, the null observability does not hold for heat equations with certain memory kernels (see, for instance, \cite{GI-ESAIMCOCV-2013, HalanayPandolfi-SCL-2012, ZhouGao-CMA-2014}). However, it was pointed out in \cite{CZZ-SICON-2017} that when an observation region moves along the time, the null observability is true for heat equations with polynomial-like memory kernels in lower order terms. This idea was absorbed in \cite{WZZ-arXiv-2023} to establish a sufficient and necessary condition on observation regions for the exact observability with observations made over time intervals away from the initial time.
  However, in these works, the observations are conducted over a time interval.

  About the studies on  the sampling observability of partial differential equations without memory,
  there have been much literature. Here,
  we would like to mention the following works: The papers \cite[Theorem 1]{WangLiYinGuoXu-2011} and \cite[Theorem 2.2]{QinWang-JDE-2017} investigate linear autonomous ODEs where the observation is made at finite time instants; 
  the paper \cite[Theorem 5.2]{QinWang-JDE-2017} studies coupled heat equations where the observation is made at a finite number of time instants; papers \cite{AEWZ-JEMS-2014, PW-JEMS-2013} deal with  heat equations  where the observation is made at one time instant; the paper \cite{Zhang-CRMASP-2016} concerns Kolmogorov equations  where the observation is made at one time instant;
  papers \cite{HuangSoffer-AJM-2021, WangWangZhang-JEMS-2019} concern Schr\"{o}dinger equations where the observation is made at two time instants.

   To the best of our knowledge, there has not been any published literature on the observability for the heat equation with memory where the observation is conducted impulsively at a finite number of time instants.

Before stating our main results, we will introduce some notation and definitions. Let $A$ be a linear  operator on $L^2(\Omega)$, defined by
\begin{align*}
     A f := \Delta f, \;\; f \in H^2(\Omega)  \cap H_0^1(\Omega).
\end{align*}
Write $(\lambda_k)_{k\geq 1}$ and $(e_k)_{k\geq 1}$ for the eigenvalues and the corresponding eigenfunctions (normalized in $L^2(\Omega)$)
of $-A$.
For each $s\in \mathbb R$, we define the following Hilbert space:
\begin{align*}
    \mathcal H^s := \bigg\{
            \sum_{k\geq 1} a_k e_k  ~:~a_k \in \mathbb{R} ~(k\in \mathbb N^+)
            ~~\text{and}~~
           \sum_{k\geq 1} a_k^2\lambda_k^{s}<+\infty
        \bigg\},
\end{align*}
equipped with the inner product
\begin{align*}
    \langle f, g \rangle_{\mathcal H^s} := \sum_{k\geq 1} a_k b_k \lambda_k^s\;\;\mbox{for all}\;\;f= \sum_{k\geq 1} a_k e_k \in \mathcal H^s
    \;\;\mbox{and}\;\;g= \sum_{k\geq 1} b_k e_k \in \mathcal H^s.
\end{align*}

\begin{definition}(Two-sided sampling observability inequality)
$(i)$  Equation \eqref{20240215-OriginalEquation} is said to satisfy the two-sided sampling observability inequality, if there is
$m\in \mathbb{N}^+$, $\{t_j\}_{j=1}^m\subset \mathbb{R}^+$ and nonempty open subsets  $\omega_1,\omega_2,\ldots,\omega_m $ of $\overline{\Omega}$,
such that for some $C>0$,
\begin{align}\label{20240226-SamplingObservability}
                C^{-1} \| y_0 \|_{\mathcal H^{-4}}
                \leq
                \sum_{j=1}^m  \| y(t_j;y_0) \|_{L^2(\omega_j)}
                \leq C \| y_0 \|_{\mathcal H^{-4}}\;\;\mbox{for any}\;\;y_0 \in L^2(\Omega).
            \end{align}
$(ii)$ When \eqref{20240226-SamplingObservability} holds, we say that  the equation \eqref{20240215-OriginalEquation}
satisfies the two-sided sampling observability inequality at $\{t_j\}_{j=1}^m$ and $\{\omega_j\}_{j=1}^m$.
\end{definition}

\begin{remark}\label{remark-two-side}
$(i)$ The first inequality in \eqref{20240226-SamplingObservability} means that by measuring a solution to the equation \eqref{20240215-OriginalEquation} on $\omega_j$ at time $t_j$ for $j=1,\dots, m$, one can recover its initial datum in
$\mathcal H^{-4}$, while the second inequality in \eqref{20240226-SamplingObservability} shows that the space $\mathcal H^{-4}$ is optimal for the recovery.

$(ii)$ Here, ``sampling" means sampling at the time instants.

$(iii)$ We explain the reason for the appearance of $~\overline{\Omega}$ as follows: In our study, we will use the propagation of singularities for the equation \eqref{20240215-OriginalEquation} (which has been revealed in \cite{WZZ-JMPA-2022}). Thus, the observation region needs, in general, to cover the characteristic lines on the boundary of $\Omega$.

$(iv)$ In \eqref{20240226-SamplingObservability}, we can replace $y_0\in L^2(\Omega)$  by $y_0\in \mathcal{H}^{-4}$,
through utilizing   the standard density argument, because when $y_0 \in \mathcal H^{-4}$, we have $y(\hat t;y_0)\in L^2(\Omega)$ for each $\hat t>0$ (see   \cite[Corollary 2.2]{WZZ-arXiv-2023}  or Lemma \ref{20240226-proposition-DecompsotionOfSolutions} of the current paper).


$(v)$   By the standard duality argument, we know that when the two-sided sampling observability inequality \eqref{20240226-SamplingObservability} holds at $\{t_j\}_{j=1}^m$ and $\{\omega_j\}_{j=1}^m$, the following impulse controllability is true: when $T> \max_{1\leq j \leq m} t_j$, for each $y_0\in L^2(\Omega)$ and each $y_1\in\mathcal{H}^{4}$, there are controls $u_j(x)\in L^2(\Omega)$,  $j=1,\dots,m$, such that the solution $y$ to the equation:
\begin{eqnarray*}
\begin{cases}
    \partial_ty(t,x) -  \Delta y(t,x) + \int_0^t M(t-s) y(s,x) ds
        = \sum_{j=1}^m \delta_{T-t_j} (t) \chi_{\omega_j}(x) u_j(x), 
        \nonumber\\
        \hskip 62pt (t,x)\in(0,T)\times\Omega,\\
    y(t,x)=0, ~(t,x)\in(0,T)\times\partial\Omega,\\
    y(0,x) = y_0(x),~x\in\Omega
\end{cases}
\end{eqnarray*}
(where $\delta_{s}$ denotes the Dirac measure at $s$) satisfies that 
  $y(T,x)=y_1(x),~x\in\Omega.$
\end{remark}

\begin{definition}\label{def-back}(Backward uniqueness)
 $(i)$ Equation \eqref{20240215-OriginalEquation} is said to have the backward uniqueness property (or the backward uniqueness for short), if
there is $m\in \mathbb{N}^+$ and a sequence $\{t_j\}_{j=1}^m\subset \mathbb{R}^+$ such that when $  y_0\in L^2(\Omega)$,
\begin{align}\label{20240215-BackwardUniqueness}
    y(t_j;y_0)=0 ~\text{in}~ L^2(\Omega)
    ~\text{for each}~ j=1,\ldots,m
    ~~\Rightarrow~~
    y_0=0 ~\text{in}~ L^2(\Omega).
\end{align}
$(ii)$ When \eqref{20240215-BackwardUniqueness} holds for some $\{t_j\}_{j=1}^m\subset \mathbb{R}^+$, with  $m\in \mathbb{N}^+$, we say that
 the equation \eqref{20240215-OriginalEquation} has the backward uniqueness at $\{t_j\}_{j=1}^m$.
\end{definition}

\begin{remark}\label{remark-back}
$(i)$ We recall the backward uniqueness of the pure heat equation: if a solution $\hat{y}$ to the pure heat equation (i.e., \eqref{20240215-OriginalEquation} where $M=0$) satisfies that for some $T>0$, $\hat{y}(T,\cdot)=0$ in $L^2(\Omega)$, then $\hat{y}(0,\cdot)=0$ in $L^2(\Omega)$. 
Differing from the pure heat equation,  the equation \eqref{20240215-OriginalEquation} may have no the above backward uniqueness for some
kernel $M$ with the assumption $(\mathcal{A})$. That is why we give Definition \ref{def-back}.
In  Section \ref{FurtherStudies}, we will further study the backward uniqueness for the equation  \eqref{20240215-OriginalEquation}
with several  special kinds of $M$.

$(ii)$ For general memory kernel $M$, the backward uniqueness for the equation \eqref{20240215-OriginalEquation}
deserves to be studied in depth.

$(iii)$   The backward uniqueness plays an important role in the studies of  the inequality \eqref{20240226-SamplingObservability}. As we will see in Remark \ref{remark-sec-resu}, the backward uniqueness  for
      the equation \eqref{20240215-OriginalEquation}  at
     $\{t_j\}_{j=1}^m$ is necessary to ensure  the inequality \eqref{20240226-SamplingObservability}, regardless of how the sequence $\{\omega_j\}_{j=1}^m$ is chosen.

\end{remark}

\begin{definition}(Geometric observation condition)
Let $m\in \mathbb{N}^+$. Let $\{t_j\}_{j=1}^m\subset \mathbb{R}^+$ and let $\omega_1,\omega_2,\ldots,\omega_m $
be open subsets of  $\overline{\Omega}$.

$(i)$ We say that $\{\omega_j\}_{j=1}^m$ satisfies the geometric observation condition for the equation \eqref{20240215-OriginalEquation}
at $\{t_j\}_{j=1}^m$, if
\begin{align}\label{20240215-GeometricObservationCondition}
             \sum_{j=1}^m |M(t_j)| ~ \chi_{\omega_j}(x)   >  0
             ~\text{for each}~  x\in \overline{\Omega}.
                  \end{align}

        $(ii)$ We say that $\{\omega_j\}_{j=1}^m$ satisfies the weak geometric observation condition for the equation \eqref{20240215-OriginalEquation}
at $\{t_j\}_{j=1}^m$, if
\begin{align}\label{20240215-NecessaryGeometricObservationCondition}
            \sum_{j=1}^m |M(t_j)| ~ \chi_{\omega_j}(x)   >  0
             ~\text{for a.e.}~  x\in \overline{\Omega}.
        \end{align}
\end{definition}
\begin{remark}
$(i)$ The gap between the geometric observation condition and the weak geometric observation condition
is quite small.

$(ii)$ Both   the geometric observation condition and the weak geometric observation condition are related to $M$, $\{t_j\}_{j=1}^m$, $\{\omega_j\}_{j=1}^m$ and $\Omega$.
\end{remark}

We now present the   first main result of this paper.

\begin{theorem}\label{20240215-FirstMainTheorem-SamplingObservability}
Let $M$ satisfy the assumption $(\mathcal{A})$. Let $m\in \mathbb{N}^+$. Let $\{t_j\}_{j=1}^m\subset \mathbb{R}^+$ and let $\omega_1,\omega_2,\ldots,\omega_m $
be nonempty and open subsets of  $\overline{\Omega}$. Suppose that the equation \eqref{20240215-OriginalEquation} has the backward uniqueness at $\{t_j\}_{j=1}^m$.
Then, each of the following statements implies the next one:
    \begin{itemize}
        \item[(i)] The sequence $\{\omega_j\}_{j=1}^m$ satisfies the geometric observation condition
       for the equation \eqref{20240215-OriginalEquation}
at $\{t_j\}_{j=1}^m$, i.e., \eqref{20240215-GeometricObservationCondition} holds.

        \item[(ii)]  Equation \eqref{20240215-OriginalEquation} satisfies
         the two-sided sampling observability inequality at  $\{t_j\}_{j=1}^m$ and $\{\omega_j\}_{j=1}^m$, i.e., \eqref{20240226-SamplingObservability} holds.

        \item[(iii)] The sequence $\{\omega_j\}_{j=1}^m$ satisfies the weak geometric observation condition
       for the equation \eqref{20240215-OriginalEquation}
at $\{t_j\}_{j=1}^m$, i.e., \eqref{20240215-NecessaryGeometricObservationCondition} holds.

\end{itemize}
\end{theorem}
\begin{remark}

$(i)$ Theorem \ref{20240215-FirstMainTheorem-SamplingObservability} shows  a sharp sufficient condition to ensure
 the two-sided sampling observability inequality \eqref{20240226-SamplingObservability}, under the assumption that  the equation \eqref{20240215-OriginalEquation} has the backward uniqueness at $\{t_j\}_{j=1}^m$.

 $(ii)$ It is worth mentioning the following: In Theorem \ref{20240215-FirstMainTheorem-SamplingObservability}, we
  assumed that the equation \eqref{20240215-OriginalEquation} has the backward uniqueness at $\{t_j\}_{j=1}^m$.
However, as we mentioned in Remark \ref{remark-back}, for some $M$, the equation \eqref{20240215-OriginalEquation} has no usual backward uniqueness. Thus, how to choose  impulse instants $\{t_j\}_{j=1}^m$ is an important issue.

 $(iii)$ When the geometric observation condition holds, we must have  $\overline{\Omega}  \subset    \cup_{j=1}^m \omega_j$ (see \eqref {20240215-GeometricObservationCondition}),
 while when the weak geometric observation condition holds, it follows from  \eqref{20240215-NecessaryGeometricObservationCondition}
 that  $\cup_{j=1}^m \omega_j$ differs from $\overline{\Omega}$ only by a  set of measure zero. Thus, it follows from Theorem \ref{20240215-FirstMainTheorem-SamplingObservability} that
 to ensure the  two-sided sampling observability inequality \eqref{20240226-SamplingObservability},  $\cup_{j=1}^m \omega_j$ must contain almost $\overline{\Omega}$. 

\end{remark}

The second main result of this paper is as follows:

\begin{theorem}\label{20240318-yubiao-MeomeryKernelDeterminesTimeInstants}
Let $M$ satisfy the assumption $(\mathcal{A})$. Let $m\in \mathbb{N}^+$. Let $\{t_j\}_{j=1}^m\subset \mathbb{R}^+$.
Then  the following two statements are equivalent:
    \begin{itemize}
        \item[(i)] There is $\{\omega_j\}_{j=1}^m$, where each $\omega_j$ is a nonempty and open subset of $\overline{\Omega}$, such that
        the equation \eqref{20240215-OriginalEquation} satisfies
         the two-sided sampling observability inequality at  $\{t_j\}_{j=1}^m$ and $\{\omega_j\}_{j=1}^m$, i.e., \eqref{20240226-SamplingObservability} holds.

        \item[(ii)]  Equation \eqref{20240215-OriginalEquation} has the backward uniqueness at $\{ t_{j} \}_{j=1}^m$, and
        it holds that
                   \begin{align}\label{20240318-yubiao-NonzerosOfMemoeryKernel}
                \sum_{j=1}^m |M(t_j)| >0.
            \end{align}
    \end{itemize}
\end{theorem}

\begin{remark}\label{remark-sec-resu}
    (i) The statement $(i)$ in Theorem \ref{20240318-yubiao-MeomeryKernelDeterminesTimeInstants} involves not only $\{t_j\}_{j=1}^m\subset \mathbb{R}^+$ and $\{\omega_j\}_{j=1}^m$, but also the memory kernel $M$.     While  the statement $(ii)$ in Theorem \ref{20240318-yubiao-MeomeryKernelDeterminesTimeInstants} is related to $M$ and $\{t_j\}_{j=1}^m\subset \mathbb{R}^+$, but  not  $\{\omega_j\}_{j=1}^m$.

    (ii) Given $M$ with the assumption $(\mathcal{A})$,  Theorem \ref{20240318-yubiao-MeomeryKernelDeterminesTimeInstants} provides a method to choose the time instants and observation regions such that the equation \eqref{20240215-OriginalEquation} satisfies the observability inequality \eqref{20240226-SamplingObservability}:
    First, we select $\{ t_{j} \}_{j=1}^m$ satisfying \eqref{20240318-yubiao-NonzerosOfMemoeryKernel} and  such that the equation \eqref{20240215-OriginalEquation} has the backward uniqueness at $\{ t_{j} \}_{j=1}^m$.
    Second, we use \eqref{20240215-GeometricObservationCondition} to choose    $\{ \omega_{j} \}_{j=1}^m$.

     $(iii)$ From Theorem \ref{20240318-yubiao-MeomeryKernelDeterminesTimeInstants}, we see that  the backward uniqueness  for
      the equation \eqref{20240215-OriginalEquation}  at
     $\{t_j\}_{j=1}^m$ is necessary to ensure  the inequality \eqref{20240226-SamplingObservability}, regardless of how 
     the sequence $\{\omega_j\}_{j=1}^m$ is chosen.
\end{remark}

The main novelties in this paper are summarized as follows:
\begin{itemize}[leftmargin=4em]
 \item[(i)] This paper appears to be the first work dealing with the observability for the heat equation with memory where the observation is conducted impulsively at a finite number of the time instants.

 \item[(ii)] Both two-sided sampling observability inequality \eqref{20240226-SamplingObservability}
    and the geometric observation condition seem to be new.

 \item[(iii)] Both  Theorems \ref{20240215-FirstMainTheorem-SamplingObservability} and \ref{20240318-yubiao-MeomeryKernelDeterminesTimeInstants} are new.

 \item[(iv)] The exploration of the backward uniqueness of \eqref{20240215-OriginalEquation}, as well as Definition 
 \ref{def-back},  seems to be new. 

\end{itemize}

The rest of the paper is organized as follows: Section \ref{RelaxedObservability} presents a relaxed observability inequality and gives
a sufficient and necessary condition to ensure it; Section \ref{UniqueContinuation} studies a unique continuation property at a finite number of  time instants; Section \ref{ProofOfMainResults} is devoted to the proofs of the main results; and Section \ref{FurtherStudies} studies the backward uniqueness at a finite number of  time instants.





\section{Relaxed observability  inequality}
\label{RelaxedObservability}

In this section, we will prove a relaxed observability inequality for the equation \eqref{20240215-OriginalEquation}.
Notice that the equation \eqref{20240215-OriginalEquation} has a unique solution for any  $y_0\in \mathcal H^s$ with  $s\in\mathbb R$. We still denote this solution by $y(\cdot;y_0)$. We write
\begin{align}\label{2.1WANG}
   \Phi(t)y_0= y(t;y_0),  ~t \geq 0.
\end{align}
Then, for each $t \geq 0$,  $\Phi(t)\in \mathcal{L}(\mathcal{H}^s)$ for any $s\in \mathbb{R}$ (see \cite[Proposition 7.1]{WZZ-JMPA-2022}).
The following lemma will play an important role in our studies. It is quoted from  \cite[Corollary 2.2]{WZZ-arXiv-2023} (see also \cite[Theorems 1.1-1.2]{WZZ-JMPA-2022}).

\begin{lemma}\label{20240226-proposition-DecompsotionOfSolutions}
(\cite[Corollary 2.2]{WZZ-arXiv-2023})
    Let $M$ satisfy the assumption $(\mathcal{A})$. Then there is an operator-valued function
      \begin{align}\label{2.1Wang}
   R\in C((0,+\infty); \mathcal L({\mathcal H}^{\hat s})) \cap L^{\infty}_{loc}([0,+\infty); \mathcal L({\mathcal H}^{\hat s})),
\end{align}
          where $\hat s$ can be any real number, such that
\begin{align}\label{20240229-DecompositionOfSolutions}
   \Phi(t) = - M(t) A^{-2} + R(t,A)  (tA)^{-3},\;\;t>0,
\end{align}
where $\Phi(t)$ is given by \eqref{2.1WANG}.
Furthermore, for any $s\in \mathbb{R}$ and $\delta\in (0,1)$, there is  $C=C(s,\delta)>0$ such that
\begin{align}\label{20240229-ContinuityWithImprovedFourOrders}
    \| y(\cdot;y_0) \|_{C([\delta,1/\delta]; \mathcal H^{s+4} )} \leq C \|y_0\|_{\mathcal H^s}
    \;\;\mbox{for each}\;\; y_0 \in \mathcal H^s.
\end{align}
\end{lemma}
\begin{remark}
Lemma \ref{20240226-proposition-DecompsotionOfSolutions} is \cite[Corollary 2.2]{WZZ-arXiv-2023} where $M$ was assumed to be analytic. However,
after carefully checking its proof, as well as the proof of \cite[Theorem 1.1]{WZZ-JMPA-2022}, we can see that it still holds for the case
when $M$ only satisfies the assumption $(\mathcal{A})$.
\end{remark}

The main result in this section is presented as follows. 
\begin{proposition}\label{20240228-proposition-RelaxedObservability}
   Suppose that  $M$ satisfies the assumption $(\mathcal{A})$. Let $m\in \mathbb{N}^+$.  Let $\{ t_j \}_{j=1}^m \subset (0,+\infty)$ and let $\omega_1,\omega_2,\ldots,\omega_m $ be  nonempty open subsets of $\overline{\Omega}$. Then, the following two statements are equivalent:
\begin{itemize}
    \item[(i)] The sequence $\{\omega_j \}_{j=1}^m$ satisfies the weak geometric observation condition for the equation \eqref{20240215-OriginalEquation} at $\{t_j \}_{j=1}^m$ (i.e., \eqref{20240215-NecessaryGeometricObservationCondition} holds);

    \item[(ii)] There is
     $C>0$ such that
    \begin{align}\label{20240215-RelaxedSamplingObservability}
    C \| y_0 \|_{\mathcal H^{-4}}
    \leq  \sum_{j=1}^m  \| \chi_{\omega_j} y(t_j;y_0) \|_{L^2(\Omega)}
        + \| y_0 \|_{\mathcal H^{-6}}\;\;\mbox{for each}\;\;y_0 \in \mathcal H^{-4}.
\end{align}
\end{itemize}
\end{proposition}

\begin{remark}
$(i)$ We call \eqref{20240215-RelaxedSamplingObservability} the relaxed observability inequality, due to the extra term
$ \| y_0 \|_{\mathcal H^{-6}}$ on the right hand side of \eqref{20240215-RelaxedSamplingObservability}. This inequality plays an important role in the proof of our main
theorems.

$(ii)$ Proposition \ref{20240228-proposition-RelaxedObservability} shows that the weak geometric observation condition is a necessary and sufficient condition on the relaxed observability inequality. While  it follows from Theorem \ref{20240215-FirstMainTheorem-SamplingObservability}
 that  the weak geometric observation condition is a necessary  condition on   the two-sided sampling observability inequality \eqref{20240226-SamplingObservability}. 
We don't know how to go from Proposition \ref{20240228-proposition-RelaxedObservability} to Theorem \ref{20240215-FirstMainTheorem-SamplingObservability} under the weak  geometric observation condition.
The main difficulty is that we do not know how to use \eqref{20240215-NecessaryGeometricObservationCondition} (instead of \eqref{20240215-GeometricObservationCondition}) to 
obtain the unique continuation property (introduced in Proposition \ref{20240228-proposition-UniqueContinuation}). While the latter plays an important role in the proof of Theorem  \ref{20240215-FirstMainTheorem-SamplingObservability}.

\end{remark}

\begin{proof}[Proof of Proposition \ref{20240228-proposition-RelaxedObservability}]
    We first show $(i)\Rightarrow(ii)$. Assume that $(i)$ holds (i.e., \eqref{20240215-NecessaryGeometricObservationCondition} is true).
    To show $(ii)$, we arbitrarily fix   $y_0\in \mathcal{H}^{-4}$. By \eqref{20240229-DecompositionOfSolutions} and \eqref{2.1WANG}, we see that
    \begin{align}\label{20240327-yubiao-MainPartOfSolutions}
        M(t) A^{-2} y_0 = - y(t;y_0) + R(t,A) (tA)^{-3} y_0,\;\;t>0.
    \end{align}
Since $R(\cdot,A)\in C((0,+\infty);\mathcal{L}(L^2(\Omega)))$ (see Lemma \ref{20240226-proposition-DecompsotionOfSolutions}),
we find from \eqref{20240327-yubiao-MainPartOfSolutions} that
\begin{align*}
   \sum_{j=1}^m \| \chi_{\omega_j} M(t_j) A^{-2} y_0 \|_{L^2(\Omega)}
    \leq \sum_{j=1}^m \| \chi_{\omega_j} y(t_j;y_0) \|_{L^2(\Omega)}
     +  C_1 \sum_{j=1}^m t_j^{-3} \| A^{-3} y_0\|_{L^2(\Omega)},
\end{align*}
where $C_1$ is a positive constant independent of $y_0$.
This yields
\begin{align}\label{2.5wang}
     \bigg\| \bigg( \sum_{j=1}^m \chi_{\omega_j} |M(t_j)| \bigg)  A^{-2} y_0  \bigg\|_{L^2(\Omega)}
    \leq \sum_{j=1}^m \| \chi_{\omega_j} y(t_j;y_0) \|_{L^2(\Omega)}
        +  C_1 \sum_{j=1}^m t_j^{-3} \| A^{-3} y_0\|_{L^2(\Omega)}.
\end{align}
Meanwhile, since for each $j\in\{1,\dots,m\}$, $x\rightarrow \sum_{j=1}^m \chi_{\omega_j}(x) |M(t_j)|$ is a step function over $\Omega$, which takes a finite number of values, we see from  \eqref{20240215-NecessaryGeometricObservationCondition}  that
\begin{align}\label{2.6wang}
    \sum_{j=1}^m \chi_{\omega_j}(x) |M(t_j)|   \geq C_2 >0
    ~\text{for a.e.}~  x \in \Omega,
\end{align}
where $C_2$ is a positive constant independent of $x$.

Since  $y_0$ was arbitrarily taken from $L^2(\Omega)$, \eqref{20240215-RelaxedSamplingObservability}
follows from \eqref{2.5wang} and \eqref{2.6wang}. This proves $(ii)$.

Next, we verify $(ii)\Rightarrow(i)$. Suppose that $(ii)$ is true (i.e.,  \eqref{20240215-RelaxedSamplingObservability} holds).
To show $(i)$, we arbitrarily fix $x_0\in \Omega$. Then for each $k\in \mathbb{N}^+$, we define $y_{0,k}\in \mathcal{H}^{-4}$ in the following manner:
\begin{align}\label{20240319-yubiao-ConstructOfTestingInitialData}
    A^{-2} y_{0,k} := |B(x_0,1/k)|^{ -\frac{1}{2} }
    \cdot \chi_{ B(x_0,1/k) \cap \Omega }.
\end{align}
Here, $B(x_0,1/k)$ denotes the open ball in $\mathbb R^n$, centered at $x_0$ and of radius $1/k$.
Then it follows from  \eqref{20240319-yubiao-ConstructOfTestingInitialData} that
\begin{align}\label{20240317-yubiao-PropertiesOfConstructedInitialData}
    \lim_{k\rightarrow +\infty} \|y_{0,k} \|_{ \mathcal H^{-4} } 
    = \lim_{k\rightarrow +\infty}   \|A^{-2} y_{0,k} \|_{ L^2(\Omega) }
    = 1.
    \end{align}
We now claim
\begin{align}\label{2.9wang}
\lim_{k\rightarrow +\infty}  y_{0,k} =0 ~\text{strongly in}~  \mathcal H^{-6}.
\end{align}
In fact, we notice that \eqref{2.9wang} is equivalent to that 
\begin{align}\label{2.10wang}
\lim_{k\rightarrow +\infty}  A^{-1}(A^{-2}y_{0,k}) =0 ~\text{strongly in}~  L^2(\Omega).
\end{align}
Since $A^{-1}$ is compact on $L^2(\Omega)$, in order to show \eqref{2.10wang}, it suffices to prove that
\begin{align*}
A^{-2}y_{0,k}\rightharpoonup 0 ~\text{weakly in}~  L^2(\Omega),\;\;\mbox{as}\;\;k\rightarrow +\infty.
\end{align*}
That is (see \eqref{20240319-yubiao-ConstructOfTestingInitialData}),
\begin{align}\label{2.12wang}
|B(x_0,1/k)|^{ -\frac{1}{2} }
    \cdot \chi_{ B(x_0,1/k) \cap \Omega }\rightharpoonup 0 ~\text{weakly in}~  L^2(\Omega),\;\;\mbox{as}\;\;k\rightarrow +\infty.
\end{align}
We now verify \eqref{2.12wang}. Without loss of generality, we can assume $x_0=0\in\Omega$. (Otherwise, we can use the translation of the coordinate system.) Let $f\in L^2(\Omega)$. Then,  when $k$ is sufficiently larger,  
\begin{align*}
   \big\langle 
        |B(0, 1/k)|^{-\frac{1}{2}} \chi_{B(0, 1/k)}, f
    \big\rangle_{L^2(\Omega)} 
  =& \int_{B(0,1/k)}  |B(0,1/k)|^{-\frac{1}{2}}f(y)dy
  \nonumber\\
  = & \int_{B(0,1)}  \big( |B(0,1)|^{-\frac{1}{2}} 
        k^{\frac{n}{2}}
    \big)
    f(z/k)  k^{-n}  dz
        \nonumber\\
  = & k^{-\frac{n}{2}}  |B(0,1)|^{-\frac{1}{2}} 
  \int_{B(0,1)}     f(z/k) dz
\end{align*}
which tends to $0$ as $k$ goes to $\infty$.
Then by \eqref{2.1WANG}, \eqref{20240229-DecompositionOfSolutions} (in Lemma \ref{20240226-proposition-DecompsotionOfSolutions}) and \eqref{2.9wang}, we have
\begin{align*}
    \lim_{k\rightarrow +\infty}  \| \chi_{\omega_j} y(t_j;y_{0,k}) \|_{L^2(\Omega)}
    =  \lim_{k\rightarrow +\infty} \Big(
            |M(t_j)|  \cdot \| \chi_{\omega_j} A^{-2} y_{0,k} \|_{L^2(\Omega)}
        \Big) \ \text{ for each}~ j=1,\ldots,m.
\end{align*}
This, along with \eqref{20240317-yubiao-PropertiesOfConstructedInitialData}, \eqref{20240215-RelaxedSamplingObservability} and \eqref{2.9wang}, yields
\begin{align}\label{2.13wang}
    C =&  C \lim_{k\rightarrow +\infty} \| y_{0,k} \|_{\mathcal H^{-4}}
    \leq   \lim_{k\rightarrow +\infty}
            \sum_{j=1}^m \| \chi_{\omega_j} y(t_j; y_{0,k}) \|_{L^2(\Omega)}
        \nonumber\\
        =& \lim_{k\rightarrow +\infty}
            \sum_{j=1}^m \Big( |M(t_j)| \cdot \| \chi_{\omega_j} A^{-2} y_{0,k} \|_{L^2(\Omega)}
        \Big),
\end{align}
where $C$ is given by \eqref{20240215-RelaxedSamplingObservability}.

At the same time, we note that for any sequence $(a_j)_{j=1}^m \subset [0,+\infty)$,
\begin{align}\label{2.14wang}
    a_1 + \cdots + a_m
    \leq   \sqrt{ m (a_1^2 + \cdots + a_m^2) }.
\end{align}
Now, it follows from \eqref{2.13wang}, along with  \eqref{20240319-yubiao-ConstructOfTestingInitialData} and \eqref{2.14wang} (with $a_j=|M(t_j)| \cdot \| \chi_{\omega_j} A^{-2} y_{0,k} \|_{L^2(\Omega)}$) that
\begin{align*}
    C^2/m  \leq \limsup_{k\rightarrow +\infty}
    {\int\hspace{-1.05em}-}_{ B(x_0,1/k) }
     \bigg( \sum_{j=1}^m \chi_{\omega_j}(x) |M(t_j)| \bigg)^2
    dx.
\end{align*}
Since $x_0$ was arbitrarily taken from $\Omega$, we see
\begin{align*}
    C/\sqrt{m}  \leq \sum_{j=1}^m \chi_{\omega_j}(x_0) |M(t_j)|
    ~\text{for a.e.}~ x_0\in \Omega,
\end{align*}
which leads to \eqref{20240215-NecessaryGeometricObservationCondition}. This shows (i).

Hence, we finish the proof of Proposition \ref{20240228-proposition-RelaxedObservability}.
\end{proof}

 \section{Unique continuation at a finite number of time instants}\label{UniqueContinuation}

This section aims to present the following unique continuation property for the equation \eqref{20240215-OriginalEquation}:

\begin{proposition}\label{20240228-proposition-UniqueContinuation}
    Let  $M$ satisfy the assumption $(\mathcal{A})$. Let $m\in \mathbb{N}^+$.  Let $\{ t_j \}_{j=1}^m \subset (0,+\infty)$ and let $\omega_1,\omega_2,\ldots,\omega_m $ be  nonempty open subsets of $\overline{\Omega}$.
    Suppose that the    equation \eqref{20240215-OriginalEquation} has the backward uniqueness at $\{ t_j \}_{j=1}^m$ (i.e., \eqref{20240215-BackwardUniqueness} holds). Assume
    that $\{ \omega_j \}_{j=1}^m$ satisfies the geometric observation condition for the    equation \eqref{20240215-OriginalEquation}
    at $\{ t_j \}_{j=1}^m$ (i.e., \eqref{20240215-GeometricObservationCondition} holds).
     Then, when $y_0 \in \mathcal H^{-4}$,
\begin{align}\label{3.1wang}
    y(t_j;y_0) =0 ~~\text{over}~~  \omega_j
    ~~\text{for each}~~ j=1,\ldots,m
    ~~\Rightarrow~~
    y_0 = 0.
\end{align}
\end{proposition}

\begin{remark}
    We do not know
    whether \eqref{3.1wang} still holds when $\{ \omega_j \}_{j=1}^m$ only satisfies the weak geometric observation condition \eqref{20240215-NecessaryGeometricObservationCondition}.
\end{remark}

\begin{proof}[Proof of Proposition \ref{20240228-proposition-UniqueContinuation}]
  We define the following linear subspace of $\mathcal H^{-4}$:
  \begin{align}\label{20240317-yubiao-InvisibleInitialData}
      V := \Big\{ y_0 \in \mathcal H^{-4} ~:~
                    \chi_{\omega_j}y(t_j;y_0) = 0, ~1\leq j \leq m
           \Big\}.
  \end{align}
Then, \eqref{3.1wang} is equivalent to that $V=\{0\}$.
To show the latter, we suppose, by
 contradiction, that
\begin{align}\label{20240317-yubiao-NontrivalInvisibleSpace}
    V \neq \{0\}.
\end{align}
We will find a contradiction to \eqref{20240317-yubiao-NontrivalInvisibleSpace} by  several steps.

\vskip 5pt
\noindent\textit{Step 1. We show that $V$ is of finite dimension in $\mathcal H^{-4}$.}
\vskip 5pt

 To this end, we arbitrarily take a bounded sequence $\{y_{0,k}\}_{k\geq1} \subset V$. Then, there is a subsequence of $(k)_{k\geq1}$,
   denoted in the same manner, such that
\begin{align}\label{3.4Ma}
y_{0,k}\rightharpoonup \hat{y} ~\text{weakly in}~  \mathcal{H}^{-4},\;\;\mbox{as}\;\;k\rightarrow +\infty,
\end{align}
and
\begin{align}\label{3.5Ma}
y_{0,k}\rightarrow \hat{y} ~\text{strongly in}~  \mathcal{H}^{-6},\;\;\mbox{as}\;\;k\rightarrow +\infty.
\end{align}

Meanwhile, it follows from 
Lemma \ref{20240226-proposition-DecompsotionOfSolutions} 
that for each $j\in\{1,\ldots,m\}$, $\Phi(t_j)\in\mathcal{L}(\mathcal{H}^{-6})$, where $\Phi$ is given by \eqref{2.1WANG}.
This, along with \eqref{3.5Ma}, leads to that for each $j\in\{1,\ldots,m\}$,
\begin{align}\label{3.6Ma}
   \Phi(t_j)y_{0,k}\rightarrow\Phi(t_j)\hat{y} ~\text{strongly in}~  \mathcal{H}^{-6},\;\;\mbox{as}\;\;k\rightarrow +\infty.
\end{align}
By \eqref{2.1WANG} and \eqref{20240317-yubiao-InvisibleInitialData}, we see that
\begin{align*}
   \chi_{\omega_j}\Phi(t_j)y_{0,k}=\chi_{\omega_j}y(t_j;y_{0,k}) = 0 \ \text{for each}~ j=1,\ldots,m,
\end{align*}
which, along with \eqref{3.6Ma}, yields that
\begin{align*}
   \chi_{\omega_j}y(t_j;\hat{y}) = \chi_{\omega_j}\Phi(t_j)\hat{y} = 0 \ \text{ for each}~ j=1,\ldots,m.
\end{align*}
Since $\hat{y}\in \mathcal{H}^{-4}$ (see \eqref{3.4Ma}), the above, along with  \eqref{20240317-yubiao-InvisibleInitialData}, leads to
\begin{align}\label{3.7Ma}
   \hat{y}  \in V   \subset  \mathcal H^{-4}.
\end{align}

At the same time, by \eqref{20240215-GeometricObservationCondition}, we can use Proposition \ref{20240228-proposition-RelaxedObservability}
to obtain \eqref{20240215-RelaxedSamplingObservability}.
Now, it follows from  \eqref{20240215-RelaxedSamplingObservability},  \eqref{3.7Ma}, \eqref{20240317-yubiao-InvisibleInitialData} and \eqref{3.5Ma} that
\begin{align*}
    y_{0,k}\rightarrow \hat{y} ~\text{ strongly in}~  \mathcal{H}^{-4},\;\;\mbox{as}\;\;k\rightarrow +\infty.
\end{align*}
Thus, we have proven that any bounded sequence in the linear subspace $V\subset \mathcal H^{-4}$ has a convergent subsequence in $\mathcal H^{-4}$, which implies that
the  $V$ is of finite dimension (see, for example, \cite[Theorem 6 on page 43]{Peter-2002}).

\vskip 5pt
\noindent\textit{Step 2. We show that $AV\subset V$.}
\vskip 5pt

 For this purpose, we arbitrarily fix $y_0 \in V$. Two facts are given in order. First, it follows from \eqref{20240317-yubiao-InvisibleInitialData} that
\begin{align}\label{20240317-yubiao-InformationOfInvisibleInitialData}
    y_0 \in \mathcal H^{-4}
    ~\text{and}~
    \chi_{\omega_j} y(t_j;y_0) =0,\;j=1,\ldots,m.
\end{align}
Second, since $\{ \omega_j \}_{j=1}^m$ satisfies the geometric observation condition for the    equation \eqref{20240215-OriginalEquation}
    at $\{ t_j \}_{j=1}^m$, it follows from \eqref{20240215-GeometricObservationCondition}  that
\begin{align}\label{20240317-yubiao-CoveringWholePhysicalDomain}
     \overline{\Omega}  \subset    \cup_{j \in J_1}  \omega_j
     ~~\text{with}~~ 
     J_1 :=  \big\{ 
        j \in \mathbb N^+ \cap [1,m]  ~:~  M(t_j) \neq  0  
    \big\}. 
\end{align}

We now claim
\begin{align}\label{3.7wang}
A^{-2}y_0\in \mathcal{H}^2,\;\mbox{i.e.}, \;Ay_0\in \mathcal{H}^{-4}.
\end{align}
%
Indeed, since $y_0 \in \mathcal H^{-4}$, we have $A^{-3}y_0\in \mathcal{H}^2$. Then by  \eqref{2.1Wang} (where $\hat s=2$)
in Lemma \ref{20240226-proposition-DecompsotionOfSolutions}, we see that
\begin{align*}
   R(t_j;A)(t_jA)^{-3}y_0\in \mathcal{H}^2,\;j=1,\dots,m,
\end{align*}
which, together with \eqref{2.1WANG} and \eqref{20240229-DecompositionOfSolutions} (in
Lemma \ref{20240226-proposition-DecompsotionOfSolutions}), yields
\begin{align}\label{3.9wang}
\rho_j:= y(t_j;y_0) + M(t_j) A^{-2} y_0\in \mathcal{H}^2 = H^2(\Omega) \cap H_0^1(\Omega),
~j=1\dots,m.
\end{align}
From \eqref{3.9wang} and  the second equality in \eqref{20240317-yubiao-InformationOfInvisibleInitialData}, we find
\begin{align}\label{3.10wang}
\rho_j|_{\omega_j}=  \big( M(t_j) A^{-2} y_0 \big)|_{\omega_j},\;j=1\dots,m.
\end{align}
At the same time, because of \eqref{20240317-yubiao-CoveringWholePhysicalDomain}, by the partition of unity, there are functions $\phi_j$ ($j\in J_1$) in $C_0^{\infty}(\mathbb R^n)$ such that
\begin{align*}
    \overline{\Omega}   \cap  \text{supp}\, \phi_j    \subset \subset  \omega_j
    ~(j \in J_1)
    ~~\text{and}~~
    \sum_{j\in J_1}  \phi_j =1 
    ~~\text{over}~~  \overline{\Omega}.
\end{align*}
Since $M(t_j)\neq 0$ for $j\in J_1$, the above, together with \eqref{3.10wang} and \eqref{3.9wang}, implies
\begin{align*}
    A^{-2} y_0   =&  \sum_{j \in J_1}   \phi_j  A^{-2} y_0
    = \sum_{j \in J_1}   \phi_j M(t_j)^{-1}  \rho_j 
    \nonumber\\
    =& \sum_{j \in J_1}   M(t_j)^{-1}  (\phi_j  \rho_j)
    \in H^2(\Omega) \cap H_0^1(\Omega) 
    = \mathcal H^2,
\end{align*}
which leads to \eqref{3.7wang}.



Finally, since 
\begin{align*}
  y(t_j;Ay_0)=Ay(t_j;y_0),
\end{align*}
we see from the second equality in 
\eqref{20240317-yubiao-InformationOfInvisibleInitialData}   that
\begin{align*}
y(t_j;Ay_0)|_{\omega_j }  = \big( A y(t_j;y_0) \big)|_{\omega_j }=0, \;j=1,\dots,m,
\end{align*}
which, together with \eqref{3.7wang} and \eqref{20240317-yubiao-InvisibleInitialData},  leads to $Ay_0\in V$. Hence $AV\subset V$.

\vskip 5pt
\noindent\textit{Step 3. We find a contradiction to \eqref{20240317-yubiao-NontrivalInvisibleSpace}.}
\vskip 5pt

By Steps 1-2, we see that  $A|_{V}$ is a linear operator from the finite dimensional subspace $V$ to itself. Thus, it follows from \eqref{20240317-yubiao-NontrivalInvisibleSpace} that  $-A|_{V}$ has at least one eigenvalue $\lambda$. Write $e$ for the corresponding
eigenfunction. It is clear that 
\begin{align}\label{3.11wang}
e\in V\setminus\{0\}.
\end{align}
Then, by the spectral method, there is a function $\xi(\cdot)$ over $[0,+\infty)$ such that
\begin{align*}
y(t;e)=\xi(t)e,\;\;t\geq 0,
\end{align*}
which leads to
\begin{align}\label{3.13wang}
y(t_j;e)=\xi(t_j)e,\;\;j=1,\dots,m.
\end{align}

Now we claim that there is $j_0\in\{1,\dots,m\}$ such that
\begin{align}\label{3.14wang}
\xi(t_{j_0})\neq 0.
\end{align}
By contradiction, we suppose that the above is not true. Then we have
\begin{align}\label{3.15wang}
\xi(t_{j})=0,\;\;j=1,\dots, m.
\end{align}
Then by \eqref{3.13wang} and \eqref{3.15wang}, we find
\begin{align*}
y(t_j;e)=0,\;\;j=1,\dots, m.
\end{align*}
Since $e\in L^2(\Omega)$, the above, along with the backward uniqueness \eqref{20240215-BackwardUniqueness}, implies that $e=0$, which contradicts to
\eqref{3.11wang}. So \eqref{3.14wang} is true.

Since $e\in V$, it follows from \eqref{20240317-yubiao-InvisibleInitialData} that
$$
\chi_{\omega_{j_0}}y(t_{j_0};e)=0,
$$
which, along with \eqref{3.13wang}, leads to
$$
\chi_{\omega_{j_0}}\xi(t_{j_0})e=0.
$$
With \eqref{3.14wang}, this shows
$
\chi_{\omega_{j_0}}e=0$.
This, along with the unique continuation property of the eigenfunctions of $-A$ (see, for instance, \cite{Lin-CPAM-1990}),  leads to
$e=0$, which contradicts \eqref{3.11wang}.

Thus, we complete the proof of Proposition \ref{20240228-proposition-UniqueContinuation}.
\end{proof}

\section{Proof of main result}\label{ProofOfMainResults}

This section is devoted to the proofs of Theorems \ref{20240215-FirstMainTheorem-SamplingObservability} and \ref{20240318-yubiao-MeomeryKernelDeterminesTimeInstants}.

\begin{proof}[Proof of Theorem \ref{20240215-FirstMainTheorem-SamplingObservability}]
First of all, we recall $(iv)$ of      Remark \ref{remark-two-side}. We organize the proof by two steps.

\vskip 5pt
\noindent {\it Step 1. We prove  $(i)\Rightarrow (ii)$.}
\vskip 5pt

The second inequality in \eqref{20240226-SamplingObservability} follows from \eqref{20240229-ContinuityWithImprovedFourOrders} with $s=-4$
at once. 
We now prove the first inequality in \eqref{20240226-SamplingObservability}.
By  contradiction, we suppose that  $(i)$ holds (i.e., \eqref{20240215-GeometricObservationCondition} is true), but
the first inequality in \eqref{20240226-SamplingObservability} fails.
Then, there is a sequence $(y_{0,k})_{k\geq 1} \subset L^2(\Omega)$ such that
\begin{align}\label{20240228-AssumptionInProofOfFirstMainTheorem}
    \|y_{0,k}\|_{\mathcal H^{-4}} = 1,\;\; k\in \mathbb N^+
\end{align}
and  
\begin{align}\label{4.2Wang}
        \lim_{k\rightarrow +\infty}
    \sum_{j=1}^m \| \chi_{\omega_j} y(t_j;y_{0,k}) \|_{L^2(\Omega)} = 0. 
\end{align}
By \eqref{20240228-AssumptionInProofOfFirstMainTheorem}, we can extract  a subsequence from $(y_{0,k})_{k\geq 1}$,  denoted by the same manner,
such that for some $\hat y_0 \in \mathcal H^{-4}$,
\begin{align}\label{20240228-WeakConvergenceInProofOfFirstMainTheorem}
    y_{0,k} \rightharpoonup \hat y_0 ~\text{weakly in}~  \mathcal H^{-4},\;\;\mbox{as}\;\;k\rightarrow +\infty.
\end{align}
It follows from \eqref{20240228-WeakConvergenceInProofOfFirstMainTheorem} and \eqref{20240229-DecompositionOfSolutions} that for each $j\in \{1, \ldots, m\}$, 
\begin{align*}
    y(t_j;y_{0,k}) \rightharpoonup y(t_j; \hat y_{0})
    ~\text{weakly in}~   L^2(\Omega),
    \;\;\mbox{as}\;\;
    k \rightarrow +\infty.
\end{align*}
With \eqref{4.2Wang}, the above
 yields
\begin{align}\label{4.5wang}
   \chi_{\omega_j} y(t_j; \hat y_{0}) = 0
   ~\text{in}~ L^2(\Omega)
    ~\text{for each }~ j \in \{1,\ldots,m\}.
\end{align}
Since the equation \eqref{20240215-OriginalEquation} has the backward uniqueness at $\{t_j\}_{j=1}^m$, and because the geometric observation condition \eqref{20240215-GeometricObservationCondition} holds, we can use
 Proposition \ref{20240228-proposition-UniqueContinuation} and \eqref{4.5wang} to get
 \begin{align*}
 \hat y_0=0\;\mbox{in}\; \mathcal{H}^{-4}.
 \end{align*}
 With \eqref{20240228-WeakConvergenceInProofOfFirstMainTheorem}, the above leads to
\begin{align}\label{20240228-StrongConvergenceInProofOfFirstMainTheorem}
     y_{0,k} \rightarrow 0\;\;\mbox{strongly in}\;\;
     \mathcal H^{-6},\;\mbox{as}\; k\rightarrow+\infty.
\end{align}

Now, since \eqref{20240215-GeometricObservationCondition} was assumed,
we can use Proposition \ref{20240228-proposition-RelaxedObservability} to get \eqref{20240215-RelaxedSamplingObservability}. Then it follows from \eqref{20240215-RelaxedSamplingObservability} (where $y_0=y_{0,k}$), \eqref{20240228-StrongConvergenceInProofOfFirstMainTheorem}
and \eqref{4.2Wang} that
\begin{align*}
    y_{0,k} \rightarrow 0\;\;\mbox{strongly in}\;\;
     \mathcal H^{-4},\;\mbox{as}\; k\rightarrow+\infty,
\end{align*}
which contradicts \eqref{20240228-AssumptionInProofOfFirstMainTheorem}. Therefore, the first inequality in \eqref{20240226-SamplingObservability} is true. Hence, $(ii)$ holds.

\vskip 5pt
\noindent{\it Step 2. We prove  $(ii)\Rightarrow (iii)$.}
\vskip 5pt

Assume that $(ii)$ is true. Then \eqref{20240215-RelaxedSamplingObservability} holds clearly. Thus, we can use Proposition \ref{20240228-proposition-RelaxedObservability} to get the weak geometric condition \eqref{20240215-NecessaryGeometricObservationCondition}, i.e., $(iii)$ holds.
 \vskip 5pt
Hence, we complete the proof of Theorem \ref{20240215-FirstMainTheorem-SamplingObservability}.
\end{proof}

Next, we go to prove Theorem \ref{20240318-yubiao-MeomeryKernelDeterminesTimeInstants}.

\begin{proof}[Proof of Theorem \ref{20240318-yubiao-MeomeryKernelDeterminesTimeInstants}]
We organize the proof by two steps.

\vskip 5pt
\noindent {\it Step 1. We prove  $(i)\Rightarrow (ii)$.}
\vskip 5pt

We suppose that $(i)$ holds for some open subsets $\omega_1,\dots,\omega_m$ in $\overline{\Omega}$, i.e., \eqref{20240226-SamplingObservability} holds for $\{t_j\}_{j=1}^m$ and $\{\omega_j\}_{j=1}^m$. 
First of all, the first inequality in \eqref{20240226-SamplingObservability} clearly implies that the equation \eqref{20240215-OriginalEquation} has the backward uniqueness at $\{t_j\}_{j=1}^m$.

Next, by the aforementioned backward uniqueness and \eqref{20240226-SamplingObservability}, we can apply Theorem \ref{20240215-FirstMainTheorem-SamplingObservability} to get \eqref{20240215-NecessaryGeometricObservationCondition},
from which, \eqref{20240318-yubiao-NonzerosOfMemoeryKernel} follows at once. 
Thus, we have proven $(ii)$.

\vskip 5pt
\noindent {\it Step 2. We prove  $(ii)\Rightarrow (i)$.}
\vskip 5pt

Assume that $(ii)$ is true. Let
\begin{align}\label{4.7wang}
    J := \big\{ j \in \{1,\dots,m\}  ~:~  M(t_j) \neq 0  \big\}.
\end{align}
Because of \eqref{20240318-yubiao-NonzerosOfMemoeryKernel}, the set $J$ is not empty. Thus, we can take a sequence
$\{\widehat{\omega}_j\}_{j\in J}$ of  open subsets of $\overline{\Omega}$ such that
\begin{align}\label{4.8wang}
 \overline{\Omega}  \subset  \cup_{j\in J }  \hat \omega_j.
          \end{align}
Then there is a sequence $\{\widetilde{\omega}\}_{j=1}^m$ of  open subsets of $\overline{\Omega}$ such that $\{\widehat{\omega}_j\}_{j\in J}$
is a subsequence of $\{\widetilde{\omega}\}_{j=1}^m$, i.e.,
\begin{align}\label{4.9wang}
 \{\widehat{\omega}_j\}_{j\in J}\subset \{\widetilde{\omega}\}_{j=1}^m.
          \end{align}
Now, by \eqref{4.7wang}, \eqref{4.8wang} and \eqref{4.9wang}, one can directly check that
\begin{align}\label{4.10wang}
             \sum_{j=1}^m |M(t_j)| ~ \chi_{\widetilde{\omega}_j}(x)   >  0
             ~\text{for each}~  x\in \overline{\Omega}.
                  \end{align}
                  Note that \eqref{4.10wang} means that $\{\widetilde{\omega}\}_{j=1}^m$
                  satisfies the geometric observation condition
       for the equation \eqref{20240215-OriginalEquation}
at $\{t_j\}_{j=1}^m$.

Meanwhile, we already assumed in  $(ii)$ that the backward uniqueness \eqref{20240215-BackwardUniqueness} holds.
Thus, we can apply Theorem \ref{20240215-FirstMainTheorem-SamplingObservability} and  \eqref{4.10wang} to see that
 the equation \eqref{20240215-OriginalEquation} satisfies
 the two-sided sampling observability inequality at  $\{t_j\}_{j=1}^m$ and $\{\widetilde{\omega}_j\}_{j=1}^m$, i.e.,
 $(i)$ is true.

 Hence, we complete the proof of  Theorem \ref{20240318-yubiao-MeomeryKernelDeterminesTimeInstants}.
\end{proof}

\section{Further studies on the backward uniqueness}
\label{FurtherStudies}

This section presents several results and examples on  the backward uniqueness for the equation \eqref{20240215-OriginalEquation}
in the sense of Definition  \ref{def-back}.
We recall that $(\lambda_k)_{k\geq 1}$ and $(e_k)_{k\geq 1}$ are the eigenvalues and the corresponding normalized eigenfunctions of $-A$, respectively. Then each $y_0\in L^2(\Omega)$ and the corresponding solution $y(\cdot;y_0)$ (to \eqref{20240215-OriginalEquation})
can be expressed as 
\begin{align}\label{5.3Ma}
    y_0 = \sum_{k \geq 1}  y_{0,k} e_k
    ~~\text{and}~~
    y(t;y_0) = \sum_{k \geq 1}  y_k(t) e_k,~ t\geq 0,
\end{align}
for some $(y_{0,k})_{k\geq 1} \subset l^2$ and $(y_k(\cdot))_{k\geq 1} \subset C([0,+\infty); l^2)$, respectively.
Meanwhile, for each $k\in \mathbb{N}^+$, we write $x_{k}(\cdot)$ for  the solution to the following equation:
\begin{align}\label{20230318-yubiao-ParametrizedODEwithMemory}
    \begin{cases}
        x'(t)  +  \lambda_k x(t)  +  \int_0^t M(t-s) x(s) ds =0,~ t>0,\\
        x(0)=1.
    \end{cases}
\end{align}
Then by the Galerkin method, we have that  for each $k\in \mathbb N^+$,
\begin{align}\label{20240318-yubiao-FunctionalCalculusForSolutions}
    y_k(t) = y_{0,k} x_{k}(t), ~ t> 0.
\end{align}
Thus, solving $y(\cdot;y_0)$ is equivalent to solving \eqref{20230318-yubiao-ParametrizedODEwithMemory} for each $k\in \mathbb N^+$.
However, the latter is not easy. Indeed, 
 for  a general kernel $M\in C^2([0,+\infty))$, it follows from  \cite[Lemma 3.3]{WZZ-JMPA-2022} (where $\eta=\lambda_k$) that for each $k\in \mathbb N^+$,
\begin{align}\label{20240318-yubiao-IntegralExpressionOfODEWithMemory}
    x_{k}(t) = e^{-\lambda_k t} + \int_0^t K_M(t,s) e^{-\lambda_k s} ds,~t\geq0,
\end{align}
where
\begin{align}\label{K_M(t,s)}
    K_M(t,s) :=&
        \sum_{j \geq 1} \frac{ s^j }{ j! } \underset{j}{ \underbrace{(-M)*\cdots*(-M)} } (t-s),~
    t\geq s.
\end{align}

As mentioned in $(i)$ of Remark \ref{remark-back}, 
  the equation \eqref{20240215-OriginalEquation} essentially differs  from the pure heat equation from perspective of the backward uniqueness, and the backward uniqueness of  \eqref{20240215-OriginalEquation} (given by  Definition \ref{def-back}) heavily depends on the memory kernel $M$. We now study this property for several concrete $M$.


Our studies are  
 closely related to  the zero points of $x_k$. 
Thus, for each  $k \in \mathbb{N}^+$, we define 
\begin{align}\label{5.9Wang}
  \mathcal{N}_{k} :=\{t\in \mathbb{R}^+\; :\; x_{k}(t)=0\}.
\end{align}

\begin{theorem}\label{Thm_singleton}
  The following statements are equivalent:
  
  $(i)$ For any $t_0>0$, the equation \eqref{20240215-OriginalEquation} has the backward uniqueness at $\{t_0\}$.
  
  $(ii)$ For each $k\in\mathbb{N}^+$, $\mathcal{N}_{k} =\emptyset$, where $\mathcal{N}_{k}$ is given by \eqref{5.9Wang}.
\end{theorem}

\begin{remark}
  Theorem \ref{Thm_singleton} shows that, in general, the equation \eqref{20240215-OriginalEquation} does not have the same backward uniqueness property as that of the heat equation, but for such a kernel $M$  that $\mathcal{N}_{k} =\emptyset$ for each $k\in\mathbb{N}^+$, \eqref{20240215-OriginalEquation} has the same backward uniqueness property as that of the heat equation.  
 We will see such kernel in   
  Example \ref{remark5.8-w-10-7} given later.

\end{remark}

\begin{proof}[Proof of Theorem \ref{Thm_singleton}]
We organize the proof by two steps.

\vskip 5pt
\noindent{\it Step 1. We prove $(i)\Rightarrow(ii)$.}
\vskip 5pt

By contradiction, we suppose that $(i)$ holds, but $(ii)$ is not true.
Then there is $k_0\in\mathbb{N}^+$ and $t_0\in\mathbb{R}^+$ such that $t_0\in\mathcal{N}_{k_0}$ (i.e., $x_{k_0}(t_0) =0$). This, along with \eqref{5.3Ma}, \eqref{20240318-yubiao-FunctionalCalculusForSolutions} and \eqref{5.9Wang}, yields 
\begin{align*}
    y(t_0;e_{k_0})= x_{{k_0}}(t_0)e_{k_0}=0
\end{align*}
(where $e_{k_0}$ is an eigenfunction of $-A$). This, together with $(i)$, yields  $e_{k_0}=0$, which leads to a contradiction.

\vskip 5pt
\noindent{\it Step 2. We prove $(ii)\Rightarrow(i)$.}
\vskip 5pt

We suppose that $(ii)$ holds. Then, it follows from  \eqref{5.9Wang} that for each $k\in\mathbb{N}^+$, 
\begin{align}\label{x_neq_0}
x_k(t)\neq0,\; t\geq0.
\end{align}
We arbitrarily take $t_1\in\mathbb{R}^+$.
Let $y_0 \in L^2(\Omega)$ satisfy
$y(t_1;y_0) = 0$. Then, by  \eqref{5.3Ma} and \eqref{20240318-yubiao-FunctionalCalculusForSolutions},  we see that
\begin{align*}
   y_{0,k} x_{k}(t_1)=0 \;\;\mbox{for each}\;\; k\in\mathbb{N}^+.
\end{align*}
With \eqref{x_neq_0}, the above leads to that 
$$
y_{0,k}=0  ~~\text{for each}~~ k\in\mathbb{N}^+, 
$$
 which gives $y_0=0$. So $(i)$ is true.

Hence, we finish the proof of Theorem \ref{Thm_singleton}.
\end{proof}

\begin{theorem}\label{Thm_guide}
  Let $t_0>0$. Suppose that  $\cup_{k\geq1}\mathcal{N}_k \neq \mathbb{R}^+$. Then the following assertions are equivalent:

    $(i)$ Equation \eqref{20240215-OriginalEquation} has the backward uniqueness at $\{t_0\}$.

    $(ii)$ It holds that  $t_0\in\mathbb{R}^+\setminus\cup_{k\geq1}\mathcal{N}_k$.
\end{theorem}

\begin{remark}
 For many  memory kernels, we have $\cup_{k\geq1}\mathcal{N}_k \neq \mathbb{R}^+$. For instance, when $M$ is real analytic, 
it follows from \eqref{20240318-yubiao-IntegralExpressionOfODEWithMemory}, \eqref{K_M(t,s)} and  \cite[Proposition 2.3]{WZZ-JMPA-2022} that each $x_k$ is a real analytic function. So  the set $\cup_{k\geq1}\mathcal{N}_k$ contains up to   countable points.

\end{remark}

\begin{proof}[Proof of Theorem \ref{Thm_guide}]
    We organize the proof by two steps.

\vskip 5pt
\noindent{\it Step 1. We prove $(i)\Rightarrow(ii)$.}
\vskip 5pt

By contradiction, we suppose that $(i)$ holds, but $(ii)$ is not true.
Then there is $k_0\in\mathbb{N}^+$ such that $t_0\in\mathcal{N}_{k_0}$. This, along with \eqref{5.3Ma}, \eqref{20240318-yubiao-FunctionalCalculusForSolutions} and \eqref{5.9Wang}, yields 
\begin{align*}
    y(t_0;e_{k_0})= x_{{k_0}}(t_0)e_{k_0}=0.
\end{align*}
The above, together with  $(i)$, implies that  $e_{k_0}=0$, which leads to a contradiction.

\vskip 5pt
\noindent{\it Step 2. We prove $(ii)\Rightarrow(i)$.}
\vskip 5pt

We suppose that $(ii)$ holds. Then, it follows from  \eqref{5.9Wang} that for each $k\in\mathbb{N}^+$, 
\begin{align}\label{x_t0_neq_0}
x_k(t_0)\neq0.
\end{align}
Let $y_0 \in L^2(\Omega)$ satisfy
$y(t_0;y_0) = 0$. Then, by  \eqref{5.3Ma} and \eqref{20240318-yubiao-FunctionalCalculusForSolutions},  we get that
\begin{align*}
   y_{0,k} x_{k}(t_0)=0 \;\;\mbox{for each}\;\; k\in\mathbb{N}^+.
\end{align*}
This, along with \eqref{x_t0_neq_0}, leads to that 
$$
y_{0,k}=0  ~~\text{for each}~~ k\in\mathbb{N}^+, 
$$
which gives $y_0=0$. So $(i)$ is true.

Hence, we complete the proof of Theorem \ref{Thm_guide}.
\end{proof}

\begin{theorem}\label{Thm_multipont}
Let $m \in \mathbb{N}^+$. Let $\{t_j\}_{j=1}^m\subset(0,+\infty)$. 
Then, the following statements are equivalent:

$(i)$ Equation \eqref{20240215-OriginalEquation} has the backward uniqueness at $\{t_j\}_{j=1}^m$. 

$(ii)$ For each $k\in \mathbb{N}^+$, there is  $j_k\in\{1,\dots,m\}$ such that
\begin{align*}
  t_{j_k}\notin\mathcal{N}_k.
\end{align*}
\end{theorem}

\begin{remark}
The statement $(ii)$ of 
Theorem \ref{Thm_multipont} provides us with a method to find $m\in \mathbb{N}^+$ and $\{t_j\}_{j=1}^m\subset(0,+\infty)$ such that the equation \eqref{20240215-OriginalEquation} has the backward uniqueness at $\{t_j\}_{j=1}^m$. 

\end{remark}

\begin{proof}[Proof of Theorem \ref{Thm_multipont}]
We organize the proof by two steps.

\vskip 5pt
\noindent{\it Step 1. We prove $(i)\Rightarrow(ii)$.}
\vskip 5pt

  By contradiction, we suppose that $(i)$ holds, but $(ii)$ is not true.
Then there is $k_0\in\mathbb{N}^+$ such that $t_j\in\mathcal{N}_{k_0}$, for each $j=1,\dots,m$.
This, along with \eqref{5.3Ma}, \eqref{20240318-yubiao-FunctionalCalculusForSolutions} and \eqref{5.9Wang}, yields 
\begin{align*}
    y(t_j;e_{k_0})= x_{{k_0}}(t_j)e_{k_0}=0 \;\;\mbox{for each}~ j=1,\ldots,m.
\end{align*}
With  $(i)$, the above implies that 
$e_{k_0}=0$, which leads to a contradiction.

\vskip 5pt
\noindent{\it Step 2. We prove $(ii)\Rightarrow(i)$.}
\vskip 5pt

We suppose that (ii) holds.
Then, it follows from \eqref{5.9Wang} that when $k\in\mathbb{N}^+$,
\begin{align}\label{x_neq_0_infin}
    x_k(t_{j_k})\neq0.
\end{align}
Let $y_0 \in L^2(\Omega)$ satisfy
\begin{align*}
  y(t_j;y_0) = 0  \;\;\mbox{for each}~ j=1,\ldots,m.
\end{align*}
With \eqref{5.3Ma} and \eqref{20240318-yubiao-FunctionalCalculusForSolutions}, the above yields
$$
y_{0,k} x_k(t_{j_k}) =0  ~~\text{for each}~~ k\in\mathbb{N}^+. 
$$
This, together with   \eqref{x_neq_0_infin}, implies
$$
y_{0,k} =0  ~~\text{for each}~~ k\in\mathbb{N}^+, 
$$
which gives $y_0=0$. So $(i)$ holds.

Hence, we finish the proof of Theorem \ref{Thm_multipont}.
\end{proof}

Next, we will show applications of the above theorems to  several concrete and important kernels.

\begin{example}\label{remark5.8-w-10-7}
Consider the kernel: $M(t)\leq0$ ($t\geq0$). 
By \eqref{K_M(t,s)}, we get that $K_M(t,s)\geq0$ for $t\geq s\geq0$. This, along with \eqref{20240318-yubiao-IntegralExpressionOfODEWithMemory} and \eqref{5.9Wang}, yields that $\mathcal{N}_k=\emptyset$ for each $k\in\mathbb{N}^+$.
\end{example}

\begin{corollary}
Let $M(t)\leq0$ ($t\geq0$). Then, for any $t_0\in (0,+\infty)$, the equation \eqref{20240215-OriginalEquation} has the backward uniqueness at $\{t_0\}$.
\end{corollary}

\begin{proof}
Since $\mathcal{N}_k=\emptyset$, it follows from Theorem \ref{Thm_singleton} that for any $t_0\in(0,+\infty)$, the equation \eqref{20240215-OriginalEquation} has the backward uniqueness at $\{t_0\}$.
\end{proof}

\begin{example}
   Consider the kernel: $M(t) = t$ ($t\geq 0$). Direct computations lead to
   \begin{align}\label{zero_of_t}
      \mathcal{N}_k= \left\{
        t\in\mathbb{R}^+\;|\;
        e^{ \alpha_k t} = \tilde{C_k} \sin(\beta_k t + \varphi_k) 
    \right\} 
    \;\;\mbox{for each}\;\; k\in\mathbb{N}^+,
   \end{align}
   where $\alpha_k<0$ and $\tilde{C_k},\beta_k,\varphi_k\in\mathbb{R}$. 
   
\end{example}  

\begin{corollary}
Let $M(t) = t$ ($t\geq 0$). Then for any $t_0\in\mathbb{R}^+\setminus\cup_{k\geq1}\mathcal{N}_k$ (where $\mathcal{N}_k$ is given by \eqref{zero_of_t}), the equation \eqref{20240215-OriginalEquation} has the backward uniqueness at $\{t_0\}$.
\end{corollary}

\begin{proof}
From \eqref{zero_of_t} and the fact that $\alpha_k<0$, we  see that for each $k\in\mathbb{N}^+$, $\mathcal{N}_k$ contains a countable number of points. Thus,  $\cup_{k\geq1}\mathcal{N}_k$
is a countable set, which implies that 
$\mathbb{R}^+\setminus\cup_{k\geq1}\mathcal{N}_k$
is not an empty set.
Now the desired result follows directly from Theorem \ref{Thm_guide}.
\end{proof}

\begin{example}
Let $c > 0$ and $\alpha \in \mathbb R$. Consider the kernel: $M(t) = c e^{\alpha t}$ ($t\geq0$). In this case, for each $k\in \mathbb N^+$,  the solution $x_k$ of the equation \eqref{20230318-yubiao-ParametrizedODEwithMemory} satisfies
\begin{align*}
    x^{''} + (\lambda_k - \alpha) x' + (c - \alpha \lambda_k) x = 0
    ~~\text{over}~~  \mathbb R^+;
    ~ x(0) =1,  ~ x'(0) =  - \lambda_k.
\end{align*}
In what follows, for each $k \in \mathbb N^+$,  $s_k$ and $w_k^{\pm}$ denote respectively  the discriminant and the  roots (possibly not real numbers) for the characteristic equation for the above ODE, i.e.,
\begin{align*}
    s_k  : =   (\lambda_k  +  \alpha)^2  - 4 c
    ~~\text{and}~~
    w_k^{\pm}  :=  \frac{
    - (\lambda_k - \alpha) \pm \sqrt{ s_k }
    }{2},
    ~ k\in \mathbb N^+.
\end{align*}
After computations, we obtain that for each $k \in \mathbb N^+$, 
\begin{itemize}
    \item[(a)]  When $s_k \neq 0$, 
    \begin{align*}
        x_k (t) =  \frac{ 
            (w_k^+ - \alpha) e^{ w_k^+ t }   -   (w_k^- - \alpha) e^{ w_k^- t }
        }{
            w_k^+  - w_k^-
        },
        ~ t \geq 0. 
    \end{align*}
    \item[(b)]   When $s_k = 0$, 
    \begin{align*}
        x_k (t) =  \Big[
            1  -  \frac{(\lambda_k + \alpha) t}{2}  
        \Big]
        e^{ - \frac{(\lambda_k - \alpha) t}{2} }, 
        ~ t \geq 0. 
    \end{align*}
\end{itemize}
Thus, we conclude that for each $k \in \mathbb N^+$, 
\begin{align}\label{20241106-yb-NodelSetsOfODEWithMemory}
    \mathcal{N}_k = 
    \begin{cases}
        \big\{  
            ( \omega_k^+   -  \omega_k^- )^{-1}
            \ln [ (\omega_k^-  -  \alpha)  /  (\omega_k^+  -  \alpha) ]
        \big\},  
        &~~\text{when}~~
            s_k > 0,
        \\
        \emptyset, &~~\text{when}~~
            s_k = 0 ~~\text{and}~~ \lambda_k \leq - \alpha,
        \\
        \{ 2/(\lambda_k + \alpha) \},
        &~~\text{when}~~
            s_k = 0 ~~\text{and}~~ \lambda_k > - \alpha,
        \\
        \Big\{ 
            \frac{2}{ \sqrt{-s_k} }
            \Big(
                \text{arccot}\, \frac{
                     \lambda_k + \alpha 
                }{ \sqrt{-s_k} }
                +  l\pi
            \Big)   
             \in\mathbb{R}^+  
            ~:~  l \in \mathbb{N} 
        \Big\} , 
        &~~\text{when}~~
            s_k < 0.
    \end{cases}
\end{align}





\end{example}

\begin{corollary}
Let $M(t)= ce^{\alpha t}$ ($t\geq0$), with $c > 0$ and $\alpha \in \mathbb R$. 
Then for any  $t_1, t_2 \in \mathbb R^+$, with  $0< t_2 - t_1 < \pi / \sqrt{c}$, 
  the equation \eqref{20240215-OriginalEquation} has the backward uniqueness at $\{t_j\}_{j=1}^2$.
\end{corollary}

\begin{proof}
Let $t_1, t_2 \in \mathbb R^+$ satisfy  $0< t_2 - t_1 < \pi / \sqrt{c}$.
We arbitrarily fix $k\in\mathbb{N}^+$ such that 
\begin{align}\label{5.12-w-11-15}
    c > (\lambda_k + \alpha)^2 /4. 
\end{align}
Then it follows from \eqref{20241106-yb-NodelSetsOfODEWithMemory} that the zero points of $x_k$ are periodic with the period 
\begin{align*}
    \pi /  \sqrt{ c -  (\lambda_k + \alpha)^2 /4 }
    \geq  \pi / \sqrt{c}. 
\end{align*}
Since $0< t_2 - t_1 < \pi / \sqrt{c}$, the above shows that there is $j_k\in\{1,2\}$ such that $  t_{j_k} \notin \mathcal{N}_k$.

Next, by making use of \eqref{20241106-yb-NodelSetsOfODEWithMemory} again, we see that for any $k$ (in $\mathbb{N}^+$) that does not satisfy  
\eqref{5.12-w-11-15}, $\mathcal{N}_k$ contains at most one point. 

Therefore, for any $k\in\mathbb{N}^+$, there is $j_k\in\{1,2\}$ such that $  t_{j_k} \notin \mathcal{N}_k$. This, together with Theorem \ref{Thm_multipont}, implies that the equation \eqref{20240215-OriginalEquation} has the backward uniqueness at $\{t_j\}_{j=1}^{2}$. This completes the proof. 
\end{proof}

\section{Acknowledgments}

The authors were partially supported by the National Natural Science Foundation of China under grants 12171359 and 12371450. The third author was also funded by the Humboldt Research Fellowship for Experienced Researchers program from the Alexander von Humboldt Foundation.

\end{document}